\documentclass[10pt, leqno]{article}

\usepackage[utf8x]{inputenc}
\usepackage{amsfonts}
\usepackage{amsmath}
\usepackage{amssymb}
\usepackage{amsthm}
\usepackage{mathrsfs}
\usepackage{latexsym}
\usepackage{graphicx}
\usepackage{bm}
\usepackage{inputenc}
\usepackage{enumerate}
\usepackage{enumitem}

\usepackage{hyperref}
\usepackage{graphicx}

\swapnumbers

\allowdisplaybreaks

\newtheorem{Lemma}{Lemma}[section]

\newtheorem{Theorem}[Lemma]{Theorem}
\newtheorem{Conjecture}[Lemma]{Conjecture}

\theoremstyle{definition}

\newtheorem{Definition}[Lemma]{Definition}

\theoremstyle{remark}
\newtheorem{Remark}[Lemma]{Remark}

\newtheoremstyle{proof*}
{3pt}
{3pt}
{\rmfamily}
{}
{\bfseries}
{.}
{.5em}
{\thmnote{#3}}
\theoremstyle{proof*}
\newtheorem*{proof*}{}

\DeclareMathOperator{\posBd}{\partial^{+}}
\DeclareMathOperator{\restrict}{\llcorner}

\DeclareMathOperator{\Clos}{Clos}  
\DeclareMathOperator{\Tan}{Tan}     

\DeclareMathOperator{\Nor}{Nor}
     
\DeclareMathOperator{\Der}{D}       
       
\DeclareMathOperator{\ap}{ap}

\DeclareMathOperator{\reach}{reach}

\DeclareMathOperator{\trace}{trace}

\DeclareMathOperator{\Dis}{Dis}

\DeclareMathOperator{\Sing}{Sing}
\DeclareMathOperator{\Reg}{Reg}
\DeclareMathOperator{\divergence}{div}
\DeclareMathOperator{\interior}{interior}

\newcommand{\Real}[1]{ \mathbf{R}^{#1}}
\newcommand{\Haus}[1]{ \mathscr{H}^{#1} }
\newcommand{\Leb}[1]{ \mathscr{L}^{#1} }
\newcommand{\rect}[1]{(\mathscr{H}^{#1},#1)}

\title{Uniqueness of singular convex hypersurfaces with lower bounded $k$-th mean curvature}
\author{Mario Santilli}

\begin{document}
		\maketitle
		
		\begin{abstract}
		 We obtain a sharp characterization of the Euclidean ball among all convex bodies $ K $ whose boundary has a pointwise $ k $-th mean curvature not smaller than a geometric constant at almost all normal points. This geometric constant depends only on the volume and the boundary area of $ K $. We deduce this characterization from a new isoperimetric-type inequality for arbitrary convex bodies, for which the equality is achieved uniquely by balls. This second result is proved in a more general context of generalized mean-convex sets. Finally we positively answer a question left open in \cite{MR3951441} proving a further sharp characterization of the ball among all convex bodies that are of class $ \mathcal{C}^{1,1} $ outside a singular set, whose Hausdorff dimension is suitably bounded from above.
		\end{abstract}
		
		\section{Introduction}
		
		\subsection*{Background}
The search of geometric conditions to characterize the (geodesic) ball in terms of curvature properties has always been a very active and fascinating field of research in geometry. This problem has attracted the attention of many authors. Results of Liebmann (1899) and S\"uss (1929) allow to assert that the ball is the unique smooth convex hypersurface with constant $ k $-th mean curvature. On  the other hand, one strives in convex geometry to obtain similar characterizations independently of smoothness assumptions and consequently the notion of curvature must be carefully formulated. An insightful notion was found by Federer in \cite{MR0110078} with the concept of \emph{curvature measures}. They are Radon measures supported on the boundary of the convex body and they are a central tool in the development of the theory. For a convex body $ K $ in $ \mathbf{R}^{n+1} $ we denote them by $C_0(K, \cdot), \ldots , C_n(K, \cdot) $ and we recall that $ C_n(K, \cdot) $ is the boundary measure of $ K $. Moreover if $ K $ was smooth then $ C_k(K, \cdot)$ would be the measure obtained integrating on $ \partial K $ the $ (n-k) $-th order mean curvature. Many papers deal with the uniqueness and stability of balls in terms of their curvature measures. We refer to the excellent treatise of Schneider \cite[8.5]{MR3155183} for a complete picture of the several contributions.
One remarkable result is proved by Schneider in \cite{MR522031} and asserts that if $ C_k(K, \cdot) = \lambda C_n(K, \cdot) $ for some $ k = 1, \ldots , n-1 $ and $ \lambda \in \mathbf{R} $ then $ K $ is a ball. This result represents a far reaching generalization of the result of Liebmann-S\"uss. The hypothesis in Schneider's theorem implies in particular that the curvature measure is absolutely continuous with respect to the boundary measure and this latter condition alone implies a certain degree of regularity. Burago and Kalinin in \cite{MR1402287} proved that if $ C_k(K; \cdot) \leq \lambda C_n(K, \cdot) $ then the normal cone of $ K $ has dimension at most $n-k $ at each point. Bangert in \cite{MR1688541} proved that if $ C_{n-1}(K, \cdot) \leq \lambda C_n(K, \cdot) $ then $ \partial K $ is of class $ \mathcal{C}^{1,1} $. In \cite{MR1719698} it is proved that the absolute continuity of the $k$-th curvature measure implies the existence of balls of dimension $ n-k $ that roll freely in $ K $ and in \cite{MR1916372} the absolute continuity of the curvature measures of $ K $ is related to the absolute continuity of the surface area measure of the polar body of $ K $.

More recently, Maggi and Delgadino in \cite{MR3921314} have proved that a set $ E $ of finite perimeter in $ \mathbf{R}^{n+1} $ with constant distributional mean curvature (this is equivalent to say that the set is a critical point of the isoperimetric functional) is a union of finitely many balls of the same radius, thus obtaining a far reaching generalization of the famous Aleksandrov's characterization of the sphere in \cite{MR0102114}. We remark that this hypothesis on the distributional mean curvature guarantees by Allard's theorem \cite{MR0307015} that the boundary is a smooth hypersurface outside a singular set of $ \Haus{n} $-measure zero. The regularity almost-everywhere is only one of the ingredients of the proof, which uses at its core a sophisticated argument based on maximum principles. It is interesting to notice, see \cite{MR1209126}, that for a convex body $ K $ the boundary $ \partial K $ has constant distributional mean curvature equal to $ \lambda $ if and only if $ C_{n-1}(K, \cdot) = \lambda C_n(K, \cdot) $. Henceforth for convex bodies the result in \cite{MR3921314} follows from Schneider's theorem in \cite{MR522031}. 

Besides the concept of curvature measures, another notion of curvature can be introduced for a convex body $K$. This is the \emph{pointwise curvature} introduced at the \emph{normal boundary points}. It can be defined by locally representing the boundary of a convex body around each point as a graph of a convex function and using the well known result on the twice differentiability almost everywhere of Aleksandrov to conclude that a pointwise second fundamental form exists at almost every point of $ \partial K $, which is the pointwise second-order differential of the representing function at the base point. The resulting pointwise $ k $-th mean curvature $ H_k(K, \cdot) $ is the density of the absolutely continuous part of $ C_{n-k}(K, \cdot) $. However, in contrast with the full curvature measures and the distributional mean curvature, no regularity can be deduced from bounds on the pointwise curvature. In fact, it is well known (see for instance \cite[Theorem 1]{MR3272763} and references therein) that most of the convex bodies (in the sense of Baire Categories) have zero pointwise second fundamental form at almost every point and at the same time no regularity better than $ C^1 $ holds around  \emph{each} point of the boundary; actually, if one looks at the behaviour of the Gauss map then complicated singular geometries emerge. Since the pointwise curvature entails no regularity properties,   it is interesting to understand if uniqueness and rigidity results can be still deduced from geometric conditions on this curvature.  However we are not aware of results of this type and we obtain in this paper contributions in this direction.

	\subsection*{Results of the present paper}
	
	In this paper we aim to give a new characterization of the ball among all convex bodies, based on restrictions only on the pointwise curvature. At first sight a such characterization seems to be impossible, since one can glue together two proper antipodal spherical caps of the unit sphere in $ \mathbf{R}^{n+1} $ to construct a convex body $ K $ with a large singular set and still with the pointwise $ k $-th mean curvature equal to ${n \choose k}$, for every $ k = 1, \ldots , n $. This simple example gives a natural obstruction to the characterizations of balls in terms of conditions on the pointwise curvature. Looking closer at this example it is not difficult to realize (see \ref{spherical caps}) that the following inequality holds:
	\begin{equation*}
	\frac{\Haus{n}(\partial K)}{(n+1)\Leb{n+1}(K)} > 1.
	\end{equation*}
	This inequality naturally suggests the following conjecture: 
	\begin{Conjecture}
		Suppose $ K \subseteq \mathbf{R}^{n+1} $ is a convex body, $ k = 1, \ldots , n $ and the pointwise $ k $-th mean curvature $ H_k(K, \cdot) $ of $ K $ satisfies 
			\begin{equation*}
			H_k(K, x) \geq  \bigg(\frac{\Haus{n}(\partial K)}{(n+1)\Leb{n+1}(K)}\bigg)^{k}{n \choose k} \quad \textrm{for $ \Haus{n} $ a.e.\ $ x \in \partial K $}.
			\end{equation*}
			
			Then $ K $ is a ball. 
	\end{Conjecture}
We remark that this lower bound is sharp, as the aforementioned example of the two spherical caps shows. In Theorem \ref{final theorem} we provide a positive solution, which is based on the following new isoperimetric principle for \emph{arbitrary} convex bodies.

\begin{Theorem}\label{HK intro}
	If $ K \subseteq \Real{n+1} $ is a convex body then
	\begin{equation*}
	\Leb{n+1}(K)\leq \frac{n}{n+1}\int_{\partial K}\frac{1}{H_1(K,x)} \, d\Haus{n}x.
	\end{equation*}
	If the equality holds and there exists $ q > 0 $ such that $ H_1(K,x)  \leq q $ for $ \Haus{n}$ a.e.\ $ x \in \partial K $, then $ K $ is a ball.
\end{Theorem}

If $ \partial K $ is of class $ \mathcal{C}^2 $ then Theorem \ref{HK intro} is known and it is a special case of a result proved by Ros in \cite{MR996826} and Montiel-Ros in \cite{MR1173047} for smooth mean-convex sets. These authors used this result to prove the characterization of the sphere among smooth compact surfaces with constant $ k $-th mean curvature with $ k = 1, \ldots , n $ (the case with $ k = 1 $ is the Aleksandrov characterization of the sphere \cite{MR0102114}). Actually in this paper we are able to obtain a far-reaching generalization of the result of Ros and Montiel-Ros for a class of \emph{generalized mean-convex closed sets}, which includes all convex bodies as a special case and it is of independent interest. Theorem \ref{HK intro} is obtained as a Corollary of this more general result. In order to define the notion of mean convexity for arbitrary closed sets we employ the concept of generalized principal curvatures $ \overline{\kappa}_{A,1}, \ldots , \overline{\kappa}_{A,n} $ defined on the generalized unit normal bundle $N(A)$ of a closed subset $ A $ of $ \mathbf{R}^{n+1} $ and developed in the works of Stacho \cite{MR534512}, Hug-Last-Weil \cite{MR2031455} and \cite{MR4117503}. We refer to Section \ref{section2} for all the relevant definitions. We recall that the positive boundary of a closed set $ A $ is the subset $ \partial^+A $ of points $ a \in A $ such that there exists an open ball $ B $ with $ B \cap A = \varnothing $ and $ a \in \Clos(B) $, and that $ \partial^+A $ can be covered by countably many $ n $-dimensional submanifolds of class $ \mathcal{C}^2 $. Consequently  an approximate mean curvature vector $ h $ exists at $ \Haus{n} $ almost all points of $ \partial^+A $. We can now state our main theorem. 
\begin{Theorem}\label{HK general intro}
	Suppose $ C \subset \Real{n+1} $ is a closed set such that 
	\begin{equation}\label{HK general intro: eq2}
	\sum_{i=1}^n \overline{\kappa}_{C,i}(x,u) \leq 0  \quad \textrm{for $ \Haus{n} $ a.e.\ $ (x,u) \in N(C) $}
	\end{equation}
	and $ h $ is an approximate mean curvature vector  of $ \partial^+C $.
	
	Then  
	\begin{equation}\label{HK general intro: eq1}
	\Leb{n+1}(\Real{n+1} \sim C)\leq \frac{n}{n+1}\int_{\partial^{+}C}\frac{1}{|h|} \, d\Haus{n}.
	\end{equation}
	Furthermore, if there exists $ q > 0 $ such $| h(x)| \leq q $ for $ \Haus{n} $ a.e.\ $ x \in \partial^{+}C $ and 
	\begin{equation*}
	\Leb{n+1}(\Real{n+1} \sim C)= \frac{n}{n+1}\int_{\partial^{+}C}\frac{1}{|h|} \, d\Haus{n} < \infty,
	\end{equation*}
	then $ \Real{n+1} \sim C $ is a union of finitely many disjointed open balls of radius not smaller than $ n/q $.
\end{Theorem}

In \cite[Theorem 8]{MR3921314} a weak version of the inequality \eqref{HK general intro: eq1} is proved using different methods. It asserts that if $ \mathbf{R}^{n+1} \sim C $ is a set of finite perimeter which is \emph{viscosity mean convex}, then 
\begin{equation}\label{HK weak intro}
	\Leb{n+1}(\mathbf{R}^{n+1} \sim C) \leq \lim_{s \to 0}\int_{S(C,s)} \frac{1}{H_s}\, d\Haus{n}
\end{equation}
where $ S(C,s) $ is the set of points at distance $ s $ from $ C $ and $ H_s $ is the approximate mean curvature of $ S(C,s) $ (which can be defined $ \Haus{n}$ almost everywhere since $ S(C,s) $ is second-order rectifiable). The notion of viscosity mean convexity is essentially equivalent to our condition \eqref{HK general intro: eq2}. Theorem \ref{HK general intro}  improves and extends \cite[Theorem 8]{MR3921314} in the following three fundamental aspects.
\begin{itemize}
	\item We obtain a characterization of the right hand side in terms of the approximate mean curvature of $ \partial^+ C $. This involves a careful analysis of the interactions between the curvatures of $ \partial^+C $ and the curvatures of $ S(C,s) $, which is a non-trivial issue in a singular setting.
	\item We obtain the characterization of the equality case. This is the key to obtain our characterization of the sphere in Theorem \ref{final theorem}.
	\item The proof of Theorem \ref{HK general intro} is based on computing certain integral-geometric equalities and inequalities on the unit normal bundle of $ C $, while the proof of \cite{MR3921314} is based on the study of the monotonicity properties of $\int_{S(C,s)} \frac{1}{H_s}\, d\Haus{n} $ with respect to $ s \in (0, \infty) $. Integral formulas play a key role in the analysis of the equality case. In fact we use them both to prove that the volume of the tubular neighbourhood of radius $ r $ of $ C $ is a polynomial in $ r $ and to obtain crucial informations on the curvature of the level sets $ S(C,s) $. Combining all these informations we can eventually conclude that $ C $ is a set of positive reach employing a result of Heveling-Hug-Last in \cite{MR2036332}, and that $S(C,s)$ is an umbilical $ \mathcal{C}^{1,1} $-hypersurface for all $ s $ smaller than the reach of $ C $. Therefore $S(C,s) $ must be a sphere and we get the conclusion.
\end{itemize}

We finally remark that in the recent work \cite[Theorem 1]{MR3951441} the ball is characterized to be the unique convex body $ K $ in $ \mathbf{R}^{n+1} $ such that its boundary is of class $ \mathcal{C}^2 $ outside a finite number of singular points and has constant $ k $-th mean curvature for some $ k = 1, \ldots , n-1 $ on its regular part (notice here we exclude $ k = n $). The authors leave as an open problem, see \cite[Remark 8]{MR3951441} to prove that the ball is the unique convex body whose boundary is of class $ \mathcal{C}^2 $ with constant $ k $-th mean curvature outside a singular set of vanishing $ \mathcal{H}^s $ measure with $ 1 \leq k \leq n-s $. We provide a positive solution of this conjecture in \ref{another characterization of the sphere}.

\paragraph{Additional note.} It is natural to ask if the general Theorem \ref{HK general intro} can be applied to other classes of mean-convex sets, besides the application to convex bodies considered here. In this regard it might be interesting to observe that the class of generalized mean-convex closed sets defined by condition \eqref{HK general intro: eq2} in Theorem \ref{HK general intro} includes the complementary of each time slice of a mean-convex level set flow and the complementary of each open set with finite perimeter and bounded distributional mean curvature. One may check these assertions employing the results in \cite{MR3466806} and \cite{MR4095952}. Moreover the method of proof of Theorem \ref{HK general intro} has already been used to obtain the uniqueness result for critical points of the anisotropic isoperimetric problem in \cite{rosa2019uniqueness}. 

		\section{Preliminaries}\label{section2}
		
		\subsection*{Notation and basic concepts}
		
		As a general rule, the notation and the terminology used without comments agree with \cite[pp.\ 669--676]{MR0257325}. The symbol $\sim$ denotes the difference between two sets. The symbols $ \mathbf{U}(a,r) $ and $ \mathbf{B}(a,r) $ denote the open and closed ball with centre $ a $ and radius $ r $ (\cite[2.8.1]{MR0257325}); $ \mathbf{S}^{m} $ is the $ m $ dimensional unit sphere in $ \Real{m+1} $ (\cite[3.2.13]{MR0257325}); $ \Leb{m} $ and $ \Haus{m} $ are the $ m $ dimensional Lebesgue and Hausdorff measure (\cite[2.10.2]{MR0257325}); $ \mathbf{G}(m,k) $ is the Grassmann manifold of all $ k $ dimensional subspaces in $ \Real{m} $ (\cite[1.6.2]{MR0257325}). The symbol $ \bullet $ denotes the standard inner product of $\Real{n}$. If $T$ is a linear subspace of $\Real{n}$ then $T^{\perp} = \Real{n} \cap \{ v : v \bullet u =0 \; \textrm{for $u \in T$} \}$ and $ T_\natural : \mathbf{R}^n \rightarrow \mathbf{R}^n $ is the orthogonal projection onto $ T $.
		 The boundary and the closure of a subset $ A $ of a normed vector space $ X $ are denoted by $ \partial A $ and $ \Clos A$. The symbols $ \Tan(A,a) $ and $ \Nor(A,a) $ denote the tangent and the normal cone of $ A $ at $ a $ (\cite[3.1.21]{MR0257325}) and \emph{the normal bundle of $A$} is defined as
		 \begin{equation*}
		 \Nor(A)= \{ (a,u) : a \in A, \; u \in \Nor(A,a) \}.
		 \end{equation*}
The $m$-dimensional approximate tangent cone of a measure $ \phi $ at a point $ a $ is denoted by $ \Tan^m(\phi, a)	$ (\cite[3.2.16]{MR0257325}). If $X$ and $Y$ are sets and $ Z \subseteq X \times Y $ we define
		\begin{equation*}
		Z | S = Z \cap \{ (x,y) : x \in S  \} \quad \textrm{for $ S \subseteq X $.}
	\end{equation*}
	Our terminology for rectifiable sets agrees with \cite[3.2.14]{MR0257325}. If $ M $ is a submanifold of class $ 2 $ then $ Q_M(x) $ is \emph{the second fundamental form of $ M $ at $ x $}; this is the unique symmetric bilinear form 
	\begin{equation*}
	Q_M(x) : \Tan(M,x) \times \Tan(M,x) \rightarrow \Nor(M,x)
	\end{equation*}
	such that $ Q_M(x)(u,v) \bullet \nu(x) = - \Der \nu(x)(u) \bullet v $ whenever $ u,v \in \Tan(M,x) $ and $ \nu $ is a normal vector field of class $ 1 $ defined in a neighborhood of $ x $.
		
		\subsection*{Generalized curvatures of arbitrary closed sets}
	
	The fundamental notion of curvature measures is introduced for an arbitrary convex body throughout the Steiner Formula, see \cite[4.2]{MR3155183}. It is often very useful to have explicit formulas of the curvature measures of a convex body in terms of the generalized principal curvatures defined on its unit normal bundle, see \cite[2.6]{MR3155183}. The construction of the generalized curvatures can be carried over arbitrary closed sets, see \cite{MR534512}, \cite{MR2031455} and \cite{MR4117503}. Here we explain a such construction.

		Suppose $A \subseteq \Real{n+1}$ is closed. The \emph{distance function to $ A $} is denoted by $\bm{\delta}_{A} $ and $ S(A,r)= \{x : \bm{\delta}_{A}(x)= r\} $. If $U$ is the set of all $x \in \Real{n+1}$ such that there exists a unique $a \in A$ with $|x-a| = \bm{\delta}_{A}(x)$, we define the \textit{nearest point projection onto~$A$} as the map $\bm{\xi}_{A} : U \rightarrow A $ characterised by the requirement
		\begin{equation*}
		| x- \bm{\xi}_{A}(x)| = \bm{\delta}_{A}(x) \quad \textrm{for $x \in U$}.
		\end{equation*}
We define $ U(A) = U \cap (\Real{n+1} \sim A) $. We recall from \cite[4.1]{MR0110078} that the reach of $ A $ is defined as
	\begin{equation*}
		\reach A = \inf\{\sup\{r : \mathbf{U}(a,r) \subseteq A \cup U(A) \} : a \in A   \}.
	\end{equation*}
	The functions $ \bm{\nu}_{A} $ and $ \bm{\psi}_{A} $ are defined by
		\begin{equation*}
		\bm{\nu}_{A}(z) = \bm{\delta}_{A}(z)^{-1}(z -  \bm{\xi}_{A}(z)) \quad \textrm{and} \quad \bm{\psi}_{A}(z)= (\bm{\xi}_{A}(z), \bm{\nu}_{A}(z)),
		\end{equation*}
		whenever $ z \in U(A)$. 
		
		We define the upper semicontinuous function $ \rho(A, \cdot) $ by
		\begin{equation*}
		\rho(A,x) = \sup \{t : \bm{\delta}_{A}(\bm{\xi}_{A}(x) + t (x-\bm{\xi}_{A}(x) ))=t \bm{\delta}_{A}(x)  \} \quad \textrm{for $ x \in U(A) $,}
		\end{equation*}
		and we say that $ x \in U(A) $ is a \emph{regular point of $ \bm{\xi}_{A} $} if and only if $ \bm{\xi}_{A}$ is approximately differentiable at $ x $ with symmetric approximate differential and $ \ap \lim_{y \to x} \rho(A,y) = \rho(A,x)>1$; cf.\ \cite[3.6, 3.7, 3.13]{MR4117503}. The set of regular points of $ \bm{\xi}_{A} $ is denoted by $ R(A)$. It is proved in \cite[3.14]{MR4117503} that $ \Leb{n+1}(\Real{n+1} \sim (A \cup R(A))) =0 $ and $ \bm{\xi}_{A}(x) + t(x - \bm{\xi}_{A}(x)) \in R(A) $ for every $ x \in R(A) $ and for every $ 0 < t < \rho(A,x) $. 
	
		We define \emph{the generalized unit normal bundle of $ A $} as
		\begin{equation*}
		N(A) = (A \times \mathbf{S}^{n}) \cap \{ (a,u) :  \bm{\delta}_{A}(a+su)=s \; \textrm{for some $ s > 0 $}\}, 
		\end{equation*}
		with $ N(A,a) = \{ v : (a,v) \in N(A)   \} $ for $ a \in A $. The \emph{positive boundary of $ A $} is defined by
		\begin{equation*}
		\partial^{+}A =A \cap  \{a : N(A,a) \neq \varnothing   \}.
		\end{equation*}
	The set $ N(A) $ is a countably $n$ rectifiable subset of $ \Nor(A) $ (cf.\ \cite[4.2, 4.3]{MR4117503}); however it may not have finite $ \Haus{n} $ measure inside all compact sets. The positive boundary $ \partial^+ A $ instead is countably $ \rect{n} $ rectifiable of class $ 2 $, see \cite[4.12]{MR4012808}.
		
		Next, we define
		\begin{equation*}
			R(N(A)) = \bm{\psi}_{A}[R(A)].
		\end{equation*}
		One may check (cf.\ \cite[4.5]{MR4117503}) that $ \Haus{n}(N(A) \sim R(N(A)) =0 $. If $ (a,u) \in R(N(A)) $, $ x \in R(A) $ and $ \bm{\psi}_{A}(x)= (a,u) $ we introduce (cf.\ \cite[4.7]{MR4117503})
		\begin{equation*}
		T_{A}(a,u) = \ap \Der \bm{\xi}_{A}(x)[\Real{n+1}] \in \bigcup_{m =0}^{n+1}\mathbf{G}(n+1, m)
		\end{equation*}
		and we define the symmetric bilinear form $ \overline{Q}(a,u): T_{A}(a,u) \times T_{A}(a,u) \rightarrow \Real{} $ by
		\begin{equation*}
		\overline{Q}_{A}(a,u)(\tau, \tau_{1})= \tau \bullet \ap \Der \bm{\nu}_{A}(x)(\sigma_{1}),
		\end{equation*}
		where $(\tau, \tau_{1}) \in  T_{A}(a,u) \times T_{A}(a,u) $ and $ \sigma_{1} \in \ap \Der \bm{\xi}_{A}(x)^{-1}[\tau_{1}] $. This is  a well-posed definition (cf.\ \cite[4.6, 4.8]{MR4117503}). We call $ \overline{Q}_{A}(a,u) $ \textit{the generalized second fundamental form of $ A $ at $ a $ in the direction $ u $.} If $ A $ is a submanifold of class $ 2 $, then $\overline{Q}_{A}$ is obviously related with its classical second fundamental form $ Q_A $ (see \ref{comparison of curvatures} where a such relation is investigated for arbitrary second-order rectifiable sets). Moreover, if $(a,u) \in R(N(A))$ we define \emph{the generalized principal curvatures of $ A $ at $(a,u) $} to be the numbers 
		\begin{equation*}
		\overline{\kappa}_{A,1}(a,u) \leq \ldots \leq \overline{\kappa}_{A,n}(a,u),
		\end{equation*}
		such that $\overline{\kappa}_{A,m +1}(a,u) = \infty $, $\overline{\kappa}_{A,1}(a,u), \ldots , \overline{\kappa}_{A,m}(a,u)$ are the eigenvalues of $ \overline{Q}_{A}(a,u)$ and $ m = \dim T_{A}(a,u)$. If $ 1 \leq k \leq n $ then the  \emph{generalized $ k $-th mean curvature of $ A $} is defined by
		\begin{equation*}
	\overline{H}_{A,k}= \sum_{1 \leq j_1 < \ldots < j_k \leq n} \; \frac{\prod_{l=1}^{k}\overline{\kappa}_{A, j_l}}{\prod_{i=1}^{n}(1+\overline{\kappa}_{A,i}^2)^{1/2}}
		\end{equation*}
		and 
		\begin{equation*}
		\overline{H}_{A,0}=  \frac{1}{\prod_{i=1}^{n}(1+\overline{\kappa}_{A,i}^2)^{1/2}}.
		\end{equation*}
		
		If $ K \subseteq \mathbf{R}^{n+1} $ is a convex body (i.e.\ a closed convex set with non empty interior) and $ 0 \leq k \leq n $ then we define the $ k $-th curvature measure of $ K $ by 
		\begin{equation}\label{curvature measure}
		C_k(K, B) = \int_{N(K)|B}\overline{H}_{K, n-k}\, d\Haus{n}
		\end{equation}
		whenever $ B \subseteq \mathbf{R}^{n+1} $ is a Borel subset. This a Radon measure over supported in $ \partial K $. We recall that $C_n(K, \cdot) = \Haus{n} \restrict \partial K $. We denote the absolutely continuous part of $ C_{k}(K, \cdot)$ with respect to $ \Haus{n} \restrict \partial K $ by $ C_{k}^{a}(K, \cdot) $ and its singular part by $ C_{k}^{s}(K, \cdot) $. It is known (see \cite{MR1654685}) that 
		\begin{equation*}
		 C_{k}^{s}(K,B) = \int_{\{(x,u): x \in B, \;  \overline{\kappa}_{K,n}(x,u) = \infty\}} \overline{H}_{K, n-k}\, d\Haus{n}
		\end{equation*}
		whenever $ B \subseteq \mathbf{R}^n $ is a Borel set.

		\subsection*{Approximate curvatures of second-order rectifiable sets}
	One feature of the method of this paper is the interplay between the generalized curvatures of a closed set $ A $ introduced in the previous section and the approximate curvature of the positive boundary $ \partial^+ A $. Here we explain how to introduce a concept of approximate curvature on arbitrary countably $ \rect{m} $ rectifiable sets of class $ 2 $. 
	
\begin{Definition}\label{second order rectifiability}
	Let $ m $ be a positive integer. We say that a set $ S \subseteq \Real{n+1} $ is \emph{countably $ \rect{m} $ rectifiable of class $ 2 $} if and only if there exists a countable collection $ F $ of $ m $ dimensional submanifolds of class $ 2 $ such that 
	\begin{equation*}
		\Haus{m}\big(S \sim \bigcup F\big) =0.
	\end{equation*}
	We say that $ S $ is $ \rect{m} $ rectifiable of class $ 2 $ if $ \Haus{m}(S) < \infty $.
\end{Definition}

The concept of approximate curvature, that we are going to introduce, is based on the following standard Lemma, whose proof can be easily inferred employing the concept of approximate tangent cone of a set (see \cite[3.2.16]{MR0257325}) and the notion of approximate differentiability of second order for functions (see \cite{MR3978264}). 

	\begin{Lemma}\label{agreement of approximate curvatures}
	If $ M $ and $ N $ are $ m $ dimensional submanifolds of class $ 1 $ [class $2 $] then 
	\begin{equation*}
	\Tan(M,x) = \Tan(N,x) \qquad [Q_M(x) = Q_N(x)]
	\end{equation*}
	for $ \Haus{m} $ a.e.\ $ x \in M \cap N $.
\end{Lemma}
\begin{proof}
It follows by \cite[2.10.19(4)]{MR0257325} that  
	\begin{equation*}
	\lim_{r \to 0}\frac{\Haus{m}((M \sim N) \cap \mathbf{B}(x,r))}{r^m} =0 \quad \textrm{for $ \Haus{m} $ a.e.\ $ x \in M \cap N $}.
	\end{equation*}
Then one observes that the conclusion of the Lemma is valid at each $ x \in M \cap N $ where a such density condition holds.
\end{proof}

\begin{Definition}
Let $ S \subseteq \Real{n+1} $ be an $ \Haus{m} $ measurable and countably $ \rect{m} $ rectifiable set of class $ 1 $. An $ \Haus{m} \restrict S $-measurable function $ \tau $ with values in $ \mathbf{G}(n+1,m) $ is an \emph{approximate tangent space of $ S $} if and only if for every $ m $ dimensional submanifold $ M $ of class $ 1 $, 
\begin{equation*}
	\tau(x) = \Tan(M,x) \quad \textrm{for $ \Haus{m} $ a.e.\ $ x \in M \cap S $.}
\end{equation*}
\end{Definition}
		\begin{Definition}
			Let $ S \subseteq \Real{n+1} $ be a set that is $ \Haus{m} $ measurable and countably $ \rect{m} $ rectifiable of class $ 2 $ and let $ \tau $ be an approximate tangent space of $ S $. We say that a function $ Q $ mapping $ \Haus{m} $ almost every $ a \in S $ into a symmetric bilinear form 
			\begin{equation*}
			Q(a) : \tau(a) \times \tau(a) \rightarrow \tau(a)^\perp
			\end{equation*}
			is an \emph{approximate second fundamental form of $ S $} if and only if the following two conditions are satisfied: 
			\begin{enumerate}
			\item the function mapping $ a \in S $ into $ Q(a) \circ \bigodot_{2}\tau(a)_\natural $ is $ \Haus{m}\restrict S $ measurable with values in $ \bigodot^2(\mathbf{R}^n, \mathbf{R}^n) $, 
			\item if $ M $ is an $ m $-dimensional submanifold of class $ 2 $ then
			\begin{equation*}
			Q(a) = Q_M(a) \quad \textrm{for $ \Haus{m} $ a.e.\ $ a \in M \cap S $.}
			\end{equation*}
			\end{enumerate}
		An \emph{approximate mean curvature vector of $ S $} is the trace of an approximate second fundamental form of $ S $.
		\end{Definition}
	
	\begin{Lemma}
	Suppose $ S \subseteq \Real{n+1} $ is $ \Haus{m} $ measurable and countably $ \rect{m} $ rectifiable [of class $ 2 $]. 
	
	Then there exists an approximate tangent space of $ S $ [an approximate second fundamental form of $ S $] and it is $ \Haus{m}\restrict S $ almost unique.
	\end{Lemma}
	
\begin{proof}
The asserted uniqueness is evident from the definition. Moreover, one can prove the existence as follows. Choose a countable and disjointed family $ \{R_i:i\geq 1\} $ of $ \Haus{m} $ measurable subsets of $ S $, whose union covers $ \Haus{m} $ almost all of $ S $, such that for every $ i \geq 1 $ there exists an $ m $ dimensional submanifold $M_i $ of class $ 1 $ [of class $ 2 $] containing $ R_i $. Define $\tau(x) = \Tan(M_i,x)$ [$Q(x) = Q_{M_i}(x)$] for $ x \in R_i $ and apply \ref{agreement of approximate curvatures} to check that $ \tau $ [that $Q$] defines an approximate tangent space of $ S $ [an approximate second fundamental form of $ S $].
\end{proof}

\begin{Remark}\label{approx mean curvature}
	The following fact will be useful: \emph{if $ S $ is countably $ \rect{m} $ rectifiable and $ \tau $ is an approximate tangent space of $ S $ then
	\begin{equation*}
	\tau(x) \subseteq \Tan(S,x)
	\end{equation*}
for $ \Haus{m} $ a.e.\ $ x \in S $.} To prove the statement one simply observes that if $ M $ is an $ m $ dimensional submanifold of class $ 1 $ then 
\begin{equation*}
	\Tan(M,x) = \Tan^m[\Haus{m}\restrict (S \cap M),x] \subseteq \Tan(S,x)
\end{equation*}
for $ \Haus{m} $ a.e.\ $ x \in S \cap M $.
\end{Remark}

The following theorem naturally relates the trace of the second fundamental form defined in the previous section with the approximate mean curvature vector of a second order rectifiable set. 

\begin{Theorem}\label{comparison of curvatures}
	Suppose $ A \subseteq \Real{n+1} $ is a closed set, $ m $ is a positive integer, $ S \subseteq A $ is $ \Haus{m} $ measurable and countably $ \rect{m} $ rectifiable of class $ 2 $, $ \tau $ is an  approximate tangent space of $ S $ and $ Q $ is an  approximate second fundamental form of $ S $. 
	
	Then there exists $ T \subseteq S $ with $ \Haus{m}(S \sim T) = 0 $ such that
	\begin{equation*}
	\tau(a) = T_{A}(a,u) \quad \textrm{and} \quad Q(a) \bullet u = - \overline{Q}_{A}(a,u)
	\end{equation*}
	 for $ \Haus{n} $ a.e.\ $ (a,u) \in N(A)|T $.
\end{Theorem}

\begin{proof}
Choose a countable collection $\{M_i: i \geq 1\}$ of $ m $ dimensional submanifolds of class $ 2 $ such that $ \Haus{m}(S \sim \bigcup_{i=1}^{\infty}M_i) =0 $. For every $ i \geq 1 $ let $ S_i $ be the set of $ a \in M_i \cap S $ such that
\begin{equation*}
	\tau(a) = \Tan(M_i,a) \quad \textrm{and} \quad Q(a) = Q_{M_i}(a),
\end{equation*}
and notice that $ \Haus{m}[(M_i \cap S) \sim S_i] =0 $. Moreover, employing \cite[6.1]{MR4117503}, for every $ i \geq 1 $ we select $ R_i \subseteq A \cap M_i $ such that $ \Haus{m}[( A \cap M_i) \sim R_i] =0 $, 
\begin{equation*}
\Tan(M_i,a) = T_A(a,u) \quad \textrm{and} \quad  \overline{Q}_A(a,u) = -Q_{M_i}(a) \bullet u
\end{equation*}
for $ \Haus{n} $ a.e.\ $ (a,u) \in N(A)|R_i $. Now the conclusion can be easily checked with
\begin{equation*}
 T = \bigcup_{i=1}^{\infty}R_i \cap S_i.
\end{equation*}
\end{proof}

\begin{Remark}
Notice that we do not claim that $ \Haus{n}(N(A)| (S \sim T))=0 $ and, in fact, we cannot replace $N(A)|T$ with $N(A)|S $ in the conclusion of the theorem, cf.\ \cite[6.3]{MR4117503}.
\end{Remark}

\begin{Remark}
	If $ S $ is a countably $ \rect{m} $ rectifiable set of class $ 2 $ with $ \Haus{m}(S) < \infty $, then there exists a unique $ m $-dimensional paraboloid at $ \Haus{m} $ almost every $ x \in S $, whose graph approximates $ S $ in a measure-theoretic sense,  see \cite[Definition 3.8 and Theorem 1.2]{MR3978264}. The second-order differentials of such paraboloids define an approximate second fundamental form of $ S $.
\end{Remark}

We finally recall the notion of normal point for a convex body and the associated principal curvatures. The refer to \cite[2.5]{MR3155183} for details. Suppose $ K \subseteq \mathbf{R}^{n+1} $ is a convex body, $ x \in \partial K $ and $ \eta \in \mathbf{S}^n $ such that $ N(K,x) = \{-\eta\} $. Let $ T = \Tan(\partial K ,x) $. Then there exists an open neighborhood $ U $ of $ x $ and a unique convex function $ f : T \rightarrow \mathbf{R} $ differentiable at $ T_\natural (x) $ such that $ f \geq 0 $, $ f(T_\natural(x)) =0 $, $ \Der(f(T_\natural(x))) =0 $ and 
\begin{equation*}
U \cap \partial K = U \cap \{  \chi + f(\chi)\eta :  \chi \in T  \}
\end{equation*}
Then we say that $ x $ is a \emph{normal point of $ K $} if and only if $ f $ is pointwise differentiable of order $ 2 $ at $ T_\natural(x) $ and we set 
\begin{equation*}
	\kappa_1(K,x) \leq \ldots \leq \kappa_n(K,x)
\end{equation*}
to be the eigenvalues of $ \Der f(T_\natural(x)) $. Henceforth we define the \emph{pointwise $ k $-th mean curvature of $ K $ at $ x $} as 
		\begin{equation*}
H_k(K,x)= \sum_{1 \leq j_1 < \ldots < j_k \leq n} \; \prod_{l=1}^{k}\kappa_{j_l}(K,x)
\end{equation*}
and 
\begin{equation*}
H_0(K,x) = 1.
\end{equation*}
We recall that (see \cite[Hilfssatz 3.6]{MR522031}) that 
\begin{equation}\label{absolutely continuous part of curvature meas}
C_{k}^{a}(K,B) = \int_{B \cap \partial K} H_{n-k}(K,x)\, d\Haus{n}x \quad \textrm{for every Borel set $ B \subseteq \mathbf{R}^{n+1} $}.
\end{equation}

\begin{Remark}\label{approx sff of a convex body}
Let $ K $ be a convex body of $ \mathbf{R}^{n+1} $. If $ x $ is a normal point of $ K $, $ N(K,x) = \{-\eta\} $, $ f $ the convex function representing $ \partial K $ around $ x $ then we define $ Q(x) : \Tan(\partial K,x) \times \Tan(\partial K,x) \rightarrow \Nor(\partial K,x) $ by $ \Der^2 f(T_\natural(x)) \eta $. One may check that $ \Tan(\partial K,\cdot) $ is an approximate tangent space of $ \partial K $ and $ Q $ is an approximate second fundamental form of $ \partial K $.
\end{Remark}

		
			\section{Totally umbilical $ \mathcal{C}^{1,1} $ hypersurfaces}
		A closed and connected hypersurface of class $ \mathcal{C}^{2} $ which is umbilical at every point must be a plane or a sphere. This result was proved by Hartman in \cite{MR0021395}. A simplified proof of this result appears in \cite{MR2424898}. The same techniques can be easily adapted to cover the case of hypersurfaces of class $ \mathcal{C}^{1,1} $, which is the relevant case for the purpose of the present paper. 
		\begin{Theorem}\label{umbilical surfaces}
			Suppose $ M \subseteq \Real{n+1} $ is a closed and connected $ \mathcal{C}^{1} $ hypersurface, suppose $ \eta : M \rightarrow \mathbf{S}^{n} $ is a Lipschitzian map with $ \eta(x)\in \Nor(M,x) $ for every $ x \in M $ and suppose that for $ \Haus{n} $ a.e.\ $ x \in M $ there exists $ \kappa(x) \in \Real{} $ such that
			\begin{equation}\label{umbilical surfaces: eq1}
			\Der \eta(x)(u) = \kappa(x)u \quad \textrm{ for every $ u \in \Tan(M,x) $.}
			\end{equation}
			
			Then $ M $ is an $ n $ dimensional plane or an $ n $ dimensional round sphere.
		\end{Theorem}
		
		\begin{proof}
			Claim 1: \emph{$ \kappa $ is ($ \Haus{n} $ almost equal to) a constant function on $ M $.}\\ Since $ M $ is connected, this is equivalent to prove that $ \kappa $ is locally constant around each point of $ M $. Since $ M $ locally corresponds at each point $ a \in M $ to a graph of a $ \mathcal{C}^{1,1} $ function, we exploit \eqref{umbilical surfaces: eq1} to see that it is enough to prove the following claim: \emph{if $ U_{1}, \ldots , U_{n} $ are bounded open intervals of $ \Real{} $, $ U = U_{1} \times \ldots \times U_{n} $ and $ f : U \rightarrow \Real{} $ is a $ \mathcal{C}^{1,1} $-function such that the conditions
				\begin{equation}\label{umbilical surfaces: eq2}
				\partial_{i}\big((1 + |\nabla f|^{2})^{-1/2} \partial_{j}f\big) =0 \quad \textrm{if $ i \neq j $} 
				\end{equation}
				\begin{equation}\label{umbilical surfaces: eq3}
				\partial_{i}\big((1 + |\nabla f|^{2})^{-1/2}   \partial_{i}f\big)=  \partial_{i+1}\big((1 + |\nabla f|^{2})^{-1/2}  \partial_{i+1}f\big) \quad \textrm{for $ i = 1, \ldots , n-1 $}
				\end{equation}
				hold on $ \Leb{n} $ almost all of $ U $, then $\partial_{i}\big((1 + |\nabla f|)^{-1/2}   \partial_{i}f\big)$ is constant on $ U $.} It follows from \eqref{umbilical surfaces: eq2} that for every $ i = 1, \ldots , n $ there exists a Lipschitzian function $ a_{i}: U_{i}\rightarrow \Real{} $ such that
			\begin{equation*}
			\big(	(1 + |\nabla f|)^{-1/2}   \partial_{i}f\big)(x) = a_{i}(x_{i}) \quad \textrm{for $ x \in U $;}
			\end{equation*}
			then we use \eqref{umbilical surfaces: eq3} to conclude that
			\begin{equation*}
			a'_{i}(x_{i}) = a'_{i+1}(x_{i+1}) \quad \textrm{for $ \Leb{n} $ a.e.\ $ x \in U $,}
			\end{equation*}
			whence we deduce that for each $ i $ the function $ a'_{i} $ is constant and the conclusion follows.
			
			It follows from \eqref{umbilical surfaces: eq1} and Claim 1 that there exists $ \lambda \in \Real{} $ such that 
			\begin{equation*}
			\Der \eta(x)(u) = \lambda u
			\end{equation*}
			for every $ u \in \Tan(M,x) $ and for $ \Haus{n} $ a.e.\ $ x \in M $. If $ \lambda =0 $ then $ \eta $ is constant on $ M $ and $ M $ is a plane. If $ \lambda \neq 0 $ then $ \eta - \lambda \bm{1}_{M} $ is constant on $ M $ and $ M $ is a sphere of radius $ 1/|\lambda| $.
		\end{proof}

	\section{Closed sets}	
	Suppose $ C \subseteq \Real{n+1} $ is a closed set. For each $ a \in C $ we define (see \cite[4.1, 4.2]{MR4012808}) the closed convex subset
	\begin{equation*}
	\Dis(C,a) = \{ v : |v| = \bm{\delta}_{C}(a+v)  \}
	\end{equation*}
	and we notice that $ N(C,a) = \{  v/|v|: 0 \neq v \in \Dis(C,a)  \} $. For every integer $ 0 \leq m \leq n+1 $ we define \textit{the $m$-th stratum of $C$} by
	\begin{equation*}
	C^{(m)} = C \cap \{ a : \dim \Dis(C,a) = n+1-m \};
	\end{equation*}
	this is a Borel set which is countably $ m $ rectifiable and countably $ \rect{m} $ rectifiable of class $ 2 $; see \cite[4.12]{MR4012808}. This definition agrees with \cite[5.1]{MR4117503} by \cite[4.4]{MR4012808}. It follows from Coarea formula \cite[3.2.22]{MR0257325} that
	\begin{equation*}
	C^{(m)} = C \cap  \{a : 0 < \Haus{n-m}(N(C,a)) < \infty  \} \quad \textrm{if $ m = 0, \ldots , n $, }
	\end{equation*}
	\begin{equation*}
	C^{(n+1)} = C \cap \{ a : N(C,a) = \varnothing   \}.
	\end{equation*}
	It will be helpful in the sequel to also consider
\begin{equation}\label{C^n_+}
	C^{(n)}_{+} = C \cap \{x : \Haus{0}(N(C,x)) = 1  \}.
	\end{equation}
Notice that $ \partial^{+}C = \bigcup_{m =0}^{n}C^{(m)} $; in particular $ \partial^{+}C $ is countably $ \rect{n} $ rectifiable of class $ 2 $.

The following theorem is the cornerstone of the paper. We remark that we make no assumption on the part of the boundary $ \partial C $ which lies outside $ \partial^{+}C $, and we do not even need to assume that $ \Haus{n}(\partial^{+}C) $ is positive (which, instead, follows as a conclusion of the theorem).
		
		\begin{Theorem}\label{inequality}
		Suppose $ C \subset \Real{n+1} $ is a closed set such that 
		\begin{equation}\label{inequality:2}
			\sum_{i=1}^n \overline{\kappa}_{C,i}(x,u) \leq 0  \quad \textrm{for $ \Haus{n} $ a.e.\ $ (x,u) \in N(C) $}
		\end{equation}
	and $ h $ is an approximate mean curvature vector  of $ \partial^+C $.

Then  
		\begin{equation*}
			\Leb{n+1}(\Real{n+1} \sim C)\leq \frac{n}{n+1}\int_{\partial^{+}C}\frac{1}{|h|} \, d\Haus{n}.
		\end{equation*}
Furthermore, if there exists $ q > 0 $ such $| h(x)| \leq q $ for $ \Haus{n} $ a.e.\ $ x \in \partial^{+}C $ and 
	\begin{equation*}
\Leb{n+1}(\Real{n+1} \sim C)= \frac{n}{n+1}\int_{\partial^{+}C}\frac{1}{|h|} \, d\Haus{n} < \infty,
\end{equation*}
then $ \Real{n+1} \sim C $ is a union of finitely many disjointed open balls of radius not smaller than $ n/q $.
		\end{Theorem}
	
	\begin{proof}
Assume $ C \neq \Real{n+1}	$. Since $ \overline{\kappa}_{C,n}(x,u) < \infty $ for $ \Haus{n} $ a.e.\ $ (x,u) \in N(C) $ it follows from \cite[5.3]{MR4117503} that $ \Haus{n}(N(C)|C^{(i)}) =0 $ for every $ i = 0, \ldots , n-1 $ and it follows from \cite[5.6]{MR4117503} that 
\begin{equation*}
\Haus{n}(N(C)| S \cap C^{(n)}) =0  \quad \textrm{whenever $\Haus{n}(S) =0 $.}
\end{equation*}
 Henceforth we conclude that
\begin{equation}\label{inequality:1}
\Haus{n}(N(C)|S) =0 \quad \textrm{whenever $\Haus{n}(S) =0 $.}
\end{equation}
It follows from \ref{comparison of curvatures} and \eqref{inequality:1} that 
		\begin{equation}\label{inequality: eq6}
	\trace	Q_{C}(x, \eta) = -h(x) \bullet \eta \quad \textrm{}
		\end{equation}
	for $ \Haus{n} $ a.e.\ $ (x, \eta) \in N(C) $.
		
		\textbf{Claim 1.} \emph{$ \trace Q_{C}(x, \eta) =0 $ for $ \Haus{n} $ a.e.\ $ (x, \eta) \in N(C)|[C^{(n)} \sim C^{(n)}_{+}] $.}\\
Firstly, we notice that if $ x \in C^{(n)} \sim C^{(n)}_{+} $ then there exists $ \eta \in \mathbf{S}^{n} $ such that $ N(C,x) = \{\eta, -\eta\} $ (by the convexity of $ \Dis(C,x) $). Then we define
\begin{equation*}
N_{0} = N(C)|[C^{(n)} \sim C^{(n)}_{+}],
\end{equation*}
\begin{equation*}
N_{1} = N_{0} \cap \{ (x, \eta): \textrm{the equations \eqref{inequality:2} and \eqref{inequality: eq6} hold for $(x, \eta) $}  \},
\end{equation*}
\begin{equation*}
N_{2} = N_{1} \cap \{(x, \eta): (x, -\eta) \notin N_{1}   \}.
\end{equation*}
We notice that $ \Haus{n}(N_0 \sim N_1) =0 $, whence we deduce that $ \Haus{n}(N_2) =0 $. If $(x, \eta) \in N_1 \sim N_2 $ then $(x, -\eta) \in N_1 $ and we get that
		\begin{equation*}
	\trace	Q_C(x, \eta) = -h(x) \bullet \eta = -\trace Q_{C}(x, -\eta),
		\end{equation*}
whence we deduce that $	\trace Q_{C}(x, \eta) =0$. Henceforth Claim 1 is proved.
		
Let $ \tau $ be a generalized approximate tangent space of $ \partial^+C $. As $ \dim \tau(x)^\perp = 1 $ for $ \Haus{n} $ a.e.\ $ x \in \partial^+C $, we employ \ref{approx mean curvature} to conclude that
\begin{equation*}
	N(C,x) \subseteq \Nor(C,x) \subseteq \Nor(\partial^+C,x) \subseteq \tau(x)^\perp
\end{equation*}
for $ \Haus{n} $ a.e.\ $ x \in \partial^+C $ and
		\begin{equation}\label{inequality: eq7}
			\trace	Q_{C}(x, \eta) = -h(x) \bullet \eta = -|h(x)| 
		\end{equation}
for $ \Haus{n} $ a.e.\ $ (x, \eta) \in N(C) $. Evidently, if $ h = 0 $ on a set of positive $ \Haus{n} $ measure then the asserted inequality is trivially satisfied. Henceforth, by Claim 1, we assume 
\begin{equation}\label{inequality: additional hypothesis}
 \Haus{n}[C^{(n)}\sim C^{(n)}_{+}] =0 \quad \textrm{and} \quad  h(x)  \neq  0 \quad \textrm{for $ \Haus{n} $ a.e.\ $ x \in \partial^{+}C $.} 
\end{equation}
		
Let $ \Omega = \Real{n+1}\sim C $. Define $ \eta : C^{(n)}_{+} \rightarrow \mathbf{S}^{n} $ so that $ \{\eta(z)\} = N(C,z) $ for $ z \in C^{(n)}_{+} $, then define 
	\begin{equation*}
Q = C^{(n)}_{+}	 \cap \{z : \textrm{$ h(z) \neq 0 $, $\dim T_{C}(z, \eta(z)) = n  $ and \eqref{inequality: eq7} holds with $(z, \eta(z))$} \}. 
	\end{equation*}

	\textbf{Claim 2.} \emph{If $ y \in \bm{\xi}_{C}^{-1}(Q) \cap \Omega $ then
		\begin{equation*}
		\bm{\nu}_{C}(y) = \eta(\bm{\xi}_{C}(y)) \quad \textrm{and}	\quad 	- \bm{\delta}_{C}(y)^{-1} \leq \kappa_{C,1}(\bm{\psi}_{C}(y)) < 0.
		\end{equation*}}
Since $ Q \subseteq C^{(n)}_{+} $ the first equation is clear and the inequality $- \bm{\delta}_{C}(y)^{-1} \leq \kappa_{C,1}(\bm{\psi}_{C}(y)) $ follows from \cite[4.8]{MR4117503}. Moreover,
	\begin{equation*}
	n \kappa_{C,1}(\bm{\psi}_{C}(y)) \leq \trace Q_{C}(\bm{\psi}_{C}(y))  = -|h(\bm{\xi}_{C}(y))| < 0.
	\end{equation*}
	
	\textbf{Claim 3.} \emph{$ \Leb{n+1}(\Omega \sim \bm{\xi}_{C}^{-1}(Q)) =0 $.}\\
	Firstly, one notices from \ref{comparison of curvatures} and \eqref{inequality:1} that
	\begin{equation*}
	T_{C}(x, \eta) = \tau(x) \in \mathbf{G}(n+1,n) \quad \textrm{for $ \Haus{n} $ a.e.\ $(x, \eta) \in N(C) $}
	\end{equation*}
	and deduces from \eqref{inequality: additional hypothesis} that $\Haus{n}(\partial^{+}C \sim Q)=0 $. Then we use \eqref{inequality:1} to get 
	\begin{equation*}
 \Haus{n}[N(C)|(\partial^{+}C \sim Q )] =0.
	\end{equation*}
	 Since $ \bm{\psi}_{C}\big(S(C,r) \cap U(C) \sim \bm{\xi}_{C}^{-1}(Q) \big) \subseteq N(C)|(\partial^{+}C \sim Q ) $ for every $ r > 0 $, it follows that 
	\begin{equation*}
		\Haus{n}\big[\bm{\psi}_{C}\big(S(C,r) \cap U(C) \sim \bm{\xi}_{C}^{-1}(Q) \big)\big] =0
	\end{equation*}
and, noting	\cite[3.3]{MR4117503}, we conclude that
	\begin{equation*}
	\Haus{n}\big( S(C,r) \cap U(C) \sim \bm{\xi}_{C}^{-1}(Q)) =0 \quad \textrm{for every $ r > 0 $.}
	\end{equation*}
	Since $ \Leb{n+1}( \Real{n+1} \sim (U(C) \cup C)) =0 $ (see \cite[3.2]{MR4117503}), it follows from Coarea formula that
	\begin{equation*}
	\Haus{n}(S(C,r) \sim U(C)  ) =0 \quad \textrm{for $ \Leb{1} $ a.e.\ $ r > 0 $.}
	\end{equation*}
Henceforth, using again Coarea formula we obtain
	\begin{equation*}
	\Leb{n+1}(\Omega \sim \bm{\xi}_{C}^{-1}(Q)) = \int_{0}^{\infty}\Haus{n}(S(C,r) \sim \bm{\xi}_{C}^{-1}(Q)) \, d\Leb{1}r =0,
	\end{equation*}
	which proves Claim 3.
	
Now we define
	\begin{equation*}
	Z = \big((N(C)|Q) \times \Real{}\big) \cap \{ (z, \eta, t) :  0 < t \leq -\kappa_{C, 1}(z, \eta)^{-1} \}
	\end{equation*}
	and 
	\begin{equation*}
	\phi(z, \eta, t) = z + t \eta \quad \textrm{for $(z, \eta, t) \in Z $}.
	\end{equation*}
One notices that the inclusion $ \bm{\xi}_{C}^{-1}(Q) \cap \Omega \subseteq \phi(Z) $ follows from Claim 2. Therefore, by Claim 3,
	\begin{equation*}
	\Leb{n+1}(\Omega \sim \phi(Z)) =0.
	\end{equation*}
Let $ \tau_1 $ and $ \tau_2 $ be the generalized approximate tangent spaces of $ N(C) $ and $ N(C) \times \mathbf{R} $ respectively. Noting that
	\begin{equation*}
\tau_2(z,\eta,t) = \tau_1(z,\eta)\times \Real{}
	\end{equation*}
	for $ \Haus{n+1} $ a.e.\ $(z,\eta,t) \in N(C)\times \Real{} $, we infer from \cite[4.11(1)]{MR4117503} that the $(n+1)$-dimensional approximate jabobian of $ \phi $ is given by
	\begin{equation*}
	\ap J_{n+1}\phi(z, \eta, t) = J(z,\eta)\cdot \prod_{j=1}^{n}|1 + t\kappa_{C, j}(z, \eta)|
	\end{equation*}
	for $ \Haus{n+1} $ a.e.\ $ (z, \eta,t) \in Z $, where 
	\begin{equation*}
	J(z,\eta) =  \prod_{j=1}^{n}\frac{1}{(1 + \kappa_{C, j}(z, \eta)^{2})^{1/2}}.
	\end{equation*}
	Then we apply the generalized area formula \cite[2.91]{MR1857292}, the classical inequality relating the arithmetic and geometric means of positive numbers and \cite[5.4]{MR4117503} to estimate
	\begin{flalign*}
	& \Leb{n+1}(\Omega) \\
	& \quad \leq	\Leb{n+1}( \phi(Z)  ) \\
	&\quad  \leq \int_{\phi(Z)}\Haus{0}(\phi^{-1}(y)) \, d\Leb{n+1}y \\
	& \quad = \int_{Z}\ap J_{n+1}\phi(z, \eta, t) \, d\Haus{n+1}(z, \eta,t)\\
	&\quad  \leq \int_{N(C)|Q}J(z, \eta)\int_{0}^{-\kappa_{C, 1}(z, \eta)^{-1}} \Big(1 + \frac{t}{n}\trace Q_{C}(z, \eta)\Big)^{n}dt\, d\Haus{n}(z, \eta) \\
	& \quad = \int_{Q}\int_{0}^{-\kappa_{C, 1}(z, \eta(z))^{-1}} \Big(1 - \frac{t}{n}|h(z)|\Big)^{n}dt\, d\Haus{n}z \\
	&\quad  \leq \int_{Q}\int_{0}^{n/|h(z)|}\Big(1 - \frac{t}{n}|h(z)|\Big)^{n}dt\, d\Haus{n}z \\
	&\quad  = \int_{\partial^+C}\int_{0}^{ n/|h(z)|}\Big(1 - \frac{t}{n}|h(z)|\Big)^{n}dt\, d\Haus{n}z \\
	&\quad  = \frac{n}{n+1}\int_{\partial^+C}\frac{1}{|h|}\, d\Haus{n}.
	\end{flalign*}
	This proves the first part of the theorem.
	
	We assume now that the equality sign holds in the last chain of estimates and that there exists $ q > 0 $ such that $ |h(x)| \leq q$ for $ \Haus{n} $ a.e.\ $ x \in \partial^+C $. In particular, the following three equations hold:
	\begin{equation}\label{inequality: eq1}
	\Leb{n+1}(\phi(Z) \sim \Omega ) =0,
	\end{equation}
	\begin{equation}\label{inequality: eq2}
	\Haus{0}(\phi^{-1}(y)) = 1 \quad \textrm{for $ \Leb{n+1} $ a.e.\ $ y \in \phi(Z) $,}
	\end{equation}
	\begin{equation}\label{inequality: eq3}
	-\kappa_{C,j}(z, \eta(z))^{-1} = \frac{n}{|h(z)|} \quad \textrm{for $ \Haus{n} $ a.e.\ $ z \in \partial^{+}C $ and $ j = 1, \ldots , n $.}
	\end{equation}
	Our goal is to prove that $ \Omega $ is a finite union of disjointed open balls. This conclusion will be deduced from the following two claims.
	
	\textbf{Claim 4.} \emph{$ \reach C \geq n/q $.}\\ Notice that $ 0 < |h(z)| \leq q $ for $ \Haus{n} $ a.e.\ $ z \in \partial^{+} C $. Let $ 0 < \rho < n/q $ and define
	\begin{equation*}
	Q_{\rho} =Q \cap \{ z : \rho < -\kappa_{C,1}(z, \eta(z))^{-1}  \}.
	\end{equation*}
	Since it follows from \eqref{inequality: eq3} that $ \Haus{n}(Q \sim Q_{\rho}) =0 $, we infer that 
	\begin{equation*}
	\Haus{n}(\partial^{+}C \sim Q_{\rho}) =0 \quad \textrm{and} \quad \Haus{n}[N(C)| (\partial^{+}C \sim Q_{\rho})] =0.
	\end{equation*}
Therefore, repeating verbatim the argument of Claim 3 with $ Q $ replaced by $ Q_{\rho} $, we conclude that
	\begin{equation}\label{inequality: eq4}
	\Leb{n+1}(\Omega \sim \bm{\xi}_{C}^{-1}(Q_{\rho})) =0.
	\end{equation}
	We define
	\begin{equation*}
	C_{\rho} = \{ z : \bm{\delta}_{C}(z) \leq \rho  \} \quad \textrm{and} \quad Z_{\rho} =(N(C)|Q_{\rho}) \times \{ t : 0 < t \leq \rho  \}
	\end{equation*}
	and we notice that
	\begin{equation}\label{inequality: eq5}
	\bm{\xi}_{C}^{-1}(Q_{\rho})\cap \Omega \cap C_{\rho} \subseteq \phi(Z_{\rho}) \subseteq C_{\rho}.
	\end{equation}
	Let $ f : \Real{n+1} \times \mathbf{S}^{n} \rightarrow \Real{} $ be a Borel measurable function with compact support. Then we employ the generalized Area formula \cite[2.91]{MR1857292} and the Coarea formula \cite[5.4]{MR4117503} to compute
	\begin{flalign*}
	&\int_{\Omega \cap C_{\rho}}f(\bm{\psi}_{C}(y))d\Leb{n+1}y \\
	& \quad = \int_{\Omega \cap C_{\rho} \cap \bm{\xi}_{C}^{-1}(Q_{\rho})}f(\bm{\psi}_{C}(y))d\Leb{n+1}y & \textrm{by \eqref{inequality: eq4}} \\
	& \quad = \int_{\Omega \cap C_{\rho} \cap \bm{\xi}_{C}^{-1}(Q_{\rho})}\int_{\phi^{-1}(y)}f \, d\Haus{0}\, d\Leb{n+1}y & \textrm{by \eqref{inequality: eq2}}\\
	& \quad = \int_{\phi(Z_{\rho})}\int_{\phi^{-1}(y)}f \, d\Haus{0}\, d\Leb{n+1}y & \textrm{by \eqref{inequality: eq1}, \eqref{inequality: eq4}, \eqref{inequality: eq5}} \\
	& \quad = \int_{Z_{\rho}} \ap J_{n+1}\phi(z,\eta,t)\, f(z,\eta)\, d\Haus{n+1}(z, \eta, t) &  \\
	& \quad = \int_{Q_{\rho}}f(z, \eta(z))\int_{0}^{\rho} \prod_{j=1}^{n}|1 + t\kappa_{C, j}(z, \eta(z))|\, dt \, d\Haus{n}z &  \\
	& \quad = \int_{Q_{\rho}}f(z, \eta(z))\int_{0}^{\rho}\Big(1- \frac{t}{n}|h(z)|\Big)^{n}\, dt \, d\Haus{n}z & \textrm{by \eqref{inequality: eq3}} \\
	&\quad = \int_{\partial^{+}C}f(z, \eta(z))\int_{0}^{\rho}\Big(1- \frac{t}{n}|h(z)|\Big)^{n}\, dt \, d\Haus{n}z & \\
	& \quad = \sum_{i=1}^{n+1}c_{i}(f)\rho^{i},
	\end{flalign*}
	where, for $ i = 1, \ldots , n+1 $,
	\begin{equation*}
	c_{i}(f) = \Big(-\frac{1}{n}\Big)^{i-1}\frac{n!}{i!(n-i+1)!} \int_{\partial^{+}C}f(z, \eta(z))|h(z)|^{i-1}\, d\Haus{n}z.
	\end{equation*}
	Therefore $ \reach C \geq n /q $ by \cite[Theorem 3]{MR2036332}.
	
	\textbf{Claim 5.} \emph{If $ 0 < r < n/q $ then $S(C,r)$ is a finite union of disjointed spheres.}\\ Since $ \reach C \geq n/q $ it follows from \cite[4.8]{MR0110078} that $ S(C,r) $ is a closed $ \mathcal{C}^{1} $ hypersurface in $ \Real{n+1} $ and $ \nu_{C}|S(C,r) $ is a unit normal Lipschitzian vector field over $ S(C,r) $. We define
	\begin{equation*}
	T = \posBd C \cap \{ z : 0 < |h(z)|\leq q, \; \kappa_{C, j}(z, \eta(z))= -|h(z)|/n \; \textrm{for $ j = 1, \ldots , n $}  \},
	\end{equation*}
	we notice that $ \Haus{n}(\partial^{+} C \sim T) =0 $ by \eqref{inequality: eq3} and we use \eqref{inequality:1} to infer (always arguing as in Claim 3)
	\begin{equation*}
	\Haus{n}(S(C,r) \sim \bm{\xi}_{C}^{-1}(T)) =0.
	\end{equation*}
	Moreover if $ x \in S(C,r)\cap \bm{\xi}_{C}^{-1}(T) $ then, denoting by $ \chi_{1}(x) \leq \ldots \leq \chi_{n}(x) $ the eigenvalues of $ \Der \bm{\nu}_{C}(x)| \Tan(S(C,r),x) $, we employ \cite[4.10]{MR4117503} to conclude
	\begin{equation*}
	\chi_{j}(x) = \frac{  \kappa_{C, j}(\bm{\xi}_{C}(x), \eta(\bm{\xi}_{C}(x)))}{1 + r\kappa_{C, j}(\bm{\xi}_{C}(x), \eta(\bm{\xi}_{C}(x))) } = \frac{|h(\bm{\xi}_{C}(x))|}{r|h(\bm{\xi}_{C}(x))| - n} 
	\end{equation*}
	for $ j = 1, \ldots , n $. Noting that 
	\begin{equation*}
	0 < \frac{|h(\bm{\xi}_{C}(x))|}{n-r|h(\bm{\xi}_{C}(x))|} \leq \frac{q}{n-rq}< \infty,
	\end{equation*}
	we use \ref{umbilical surfaces} to conclude that $ S(C,r) $ is a union of at most countably many spheres with radii not smaller than $ q^{-1}(n-rq) $. Since $ \Leb{n+1}(\Omega) < \infty $, there are only finitely many spheres and Claim 5 is proved.
	
	We are now ready to conclude the proof. We notice from \cite[4.20]{MR0110078} that 
	\begin{equation*}
	\partial C = \posBd C = \{ x : \dim \Nor(C,x) \geq 1     \}.
	\end{equation*}
	Since $ \reach C \geq n/q $ by Claim 4, we deduce that $ \bm{\xi}_{C}(S(C,r)) = \partial C $ for $ 0 < r < n/q $. Since $ \bm{\nu}_{C}|S(C,r) $ is a unit normal vector field over $ S(C,r) $ and
	\begin{equation*}
	\bm{\xi}_{C}(x) = x - r \bm{\nu}_{C}(x) \quad \textrm{for $ x \in S(C,r) $,}
	\end{equation*}
	the conclusion follows from Claim 5.
\end{proof}

\section{Convex bodies}

The following lemma is the key for the main result of this section.
\begin{Lemma}\label{inner Lusin condition}
Suppose $ K \subseteq \Real{n+1} $ is a convex body and $C =\Real{n+1} \sim \interior(K) $. 

Then $ \partial^{+}C = C^{(n)}_{+} $, $ N(C,x)  = - N(K,x) $ for every $ x \in \partial^+C $ and
\begin{equation}\label{inner Lusin condition : eq1}
	\Haus{n}(N(C)|S) =0 \quad \textrm{whenever $ \Haus{n}(S) =0 $.}
\end{equation}
\end{Lemma}

\begin{proof}
Fix $ x \in \partial^{+}C $, $ \nu \in N(K,x) $ and $ H = \{ x + v : v \bullet \nu \leq 0  \} $. It follows that $ K \subseteq H $ by \cite[1.3.2]{MR3155183}. One easily checks that every closed ball $ B $ with $ x \in B \subseteq K $ has to be tangent to $ \partial H $ and consequently the exterior normal of $ B $ at $ x $ is $ \nu $. Therefore $ \partial^+C = C^{(n)}_+ $ and $ N(C,x) = - N(K,x) $.

Define $ \eta : \partial^{+}C \rightarrow \mathbf{S}^{n} $ so that $ \{\eta(x)\} = N(C,x) $ for $ x \in \partial^{+}C $. For $ \epsilon > 0 $ define
\begin{equation*}
B_{\epsilon} = (\partial^{+}C) \cap  \{ x  :  \bm{\delta}_{C}(x + \epsilon \eta(x)) = \epsilon \}
\end{equation*}
\begin{equation*}
K_{\epsilon} = K \cap \{ x : \bm{\delta}_{C}(x) \geq \epsilon  \} = K \cap \{x : \mathbf{B}(x, \epsilon) \subseteq K  \}
\end{equation*}
and notice that $ K_{\epsilon} $ is a convex set (cf.\ \cite{MR3155183}). If $ x \in B_{\epsilon} $ then $ x + \epsilon \eta(x) \in K_{\epsilon} $ and
\begin{equation*}
|x - \bm{\xi}_{K_{\epsilon}}(x)| \geq \bm{\delta}_{C}(\bm{\xi}_{K_{\epsilon}}(x)) \geq \epsilon,
\end{equation*}
whence we readily infer that $ \bm{\xi}_{K_{\epsilon}}(x)=x + \epsilon \eta(x) $. It follows that
\begin{equation*}
N(C)|B_{\epsilon} = \big\{    (x, \epsilon^{-1}(\bm{\xi}_{K_{\epsilon}}(x)-x)): x \in B_{\epsilon}\big\} 
\end{equation*}
and, since $ \bm{\xi}_{K_{\epsilon}} $ is a Lipschitz map, we conclude that 
\begin{equation*}
\Haus{n}(N(C)| (S \cap B_\epsilon)) =0 \quad \textrm{whenever $ S \subseteq \mathbf{R}^{n+1} $ with $ \Haus{n}(S) =0 $.}
\end{equation*}
Finally, noting that $ \bigcup\{B_{\epsilon}: \epsilon > 0   \} = \partial^{+}C $, we obtain the conclusion in \eqref{inner Lusin condition : eq1}.
\end{proof}

\begin{Remark}\label{McMullen theorem}
It is well known that $ \Haus{n}[(\partial K) \sim (\partial^{+}C)] =0 $, cf.\ \cite{MR367810}. 
\end{Remark}

\begin{Theorem}\label{inequality convex sets}
If $ K \subseteq \Real{n+1} $ is a convex body then
\begin{equation*}
\Leb{n+1}(K)\leq \frac{n}{n+1}\int_{\partial K}\frac{1}{H_1(K,x)} \, d\Haus{n}x.
\end{equation*}
If the equality holds and there exists $ q > 0 $ such that $ H_1(K,x)  \leq q $ for $ \Haus{n}$ a.e.\ $ x \in \partial K $, then $ K $ is a ball.
\end{Theorem}

\begin{proof}
Let $ C = \mathbf{R}^{n+1} \sim \interior(K) $ and let $ Q $ be the approximate second fundamental form of $ \partial K $ defined at the normal points of $ \partial K $ (see \ref{approx sff of a convex body}). It follows from \ref{comparison of curvatures} and \ref{inner Lusin condition} that 
\begin{equation*}
T_C(x, \eta) = \Tan(\partial K,x), \qquad 	\overline{Q}_{C}(x,\eta) = -Q(x) \bullet \eta 
\end{equation*}
and
\begin{equation*}
\sum_{i=1}^n \overline{\kappa}_{C,i}(x, \eta) = - \trace(Q(x) \bullet \eta) = - H_1(K,x) \leq 0 
\end{equation*}
for $ \Haus{n} $ a.e.\ $(x, \eta) \in N(C) $. Therefore the conclusion follows from \ref{inequality}.
\end{proof}

\begin{Theorem}\label{final theorem}
Suppose $ k \geq 1 $ and $ K \subseteq \mathbf{R}^{n+1} $ is a convex body such that 
\begin{equation*}
H_k(K, x) \geq  \bigg(\frac{\Haus{n}(\partial K)}{(n+1)\Leb{n+1}(K)}\bigg)^{k}{n \choose k} \quad \textrm{for $ \Haus{n} $ a.e.\ $ x \in \partial K $.}
\end{equation*}

Then $ K $ is a ball. 
\end{Theorem}

\begin{proof}
It follows from Newton-McLaurin inequality that 
\begin{equation*}
\frac{1}{n}H_1(K,x) \geq {n \choose k}^{- \frac{1}{k}}H_k(K,x)^{\frac{1}{k}} \geq \frac{\Haus{n}(\partial K)}{(n+1)\Leb{n+1}(K)}
\end{equation*}
for $ \Haus{n} $ a.e.\ $ x \in \partial K $. 

We claim now that 
\begin{equation}\label{final theorem: eq1}
\frac{1}{n}H_1(K,x)  = \frac{\Haus{n}(\partial K)}{(n+1)\Leb{n+1}(K)}	\quad \textrm{for $ \Haus{n} $ a.e.\ $ x \in \partial K $.}
\end{equation}
By contradiction we assume that there exists $ \epsilon > 0 $ such that the set 
\begin{equation*}
B_\epsilon = \partial K \cap \Bigg\{ x : H_1(K,x) \geq (1+ \epsilon)  \frac{n\Haus{n}(\partial K)}{(n+1)\Leb{n+1}(K)}  \Bigg\}
\end{equation*}
has positive $ \Haus{n} $-measure. Then 
\begin{flalign*}
\int_{\partial K}\frac{n}{H_1(K,y)}\, d\Haus{n}y & = \int_{\partial K \sim B_\epsilon}\frac{n}{H_1(K,y)}\, d\Haus{n}y + \int_{B_\epsilon}\frac{n}{H_1(K,y)}\, d\Haus{n}y \\
& \leq \frac{(n+1)\Leb{n+1}(K)}{\Haus{n}(\partial K)}\bigg( \Haus{n}(\partial K \sim B_\epsilon) + \frac{1}{1+\epsilon}\Haus{n}(B_\epsilon)\bigg)\\
& < (n+1)\Leb{n+1}(K),
\end{flalign*}
which contradicts the inequality in \eqref{inequality convex sets}. 

We now readily infer from \eqref{final theorem: eq1} and Theorem \ref{inequality convex sets} that $ K $ is a ball. 
\end{proof}

\begin{Remark}\label{spherical caps}
Theorem \ref{final theorem} is sharp. In fact, if we consider the union of two proper antipodal spherical caps of the unit sphere, we obtain a convex body $ K $ whose pointwise $ k $-th mean curvature is constant and smaller than $\big(\frac{n\Haus{n}(\partial K)}{(n+1)\Leb{n+1}(K)}\big)^{k}{n \choose k}$. We provide the details for completeness. Let $ P \in \mathbf{G}(n+1,n) $ and $ \eta \in P^\perp $ with $ | \eta | = 1 $. For every $ 0 < \epsilon < 1 $ we define 
\begin{equation*}
	\Sigma^+_\epsilon = \mathbf{S}^{n} \cap \{ x : x \bullet \eta \geq \epsilon   \}, \quad 	\Sigma^-_\epsilon = \mathbf{S}^{n} \cap \{ x : x \bullet \eta \leq- \epsilon   \}
\end{equation*}
\begin{equation*}
P^+_\epsilon = \mathbf{B}(0,1) \cap \{ x : x \bullet \eta = \epsilon   \}, \quad 	P^-_\epsilon = \mathbf{B}(0,1) \cap \{ x : x \bullet \eta =- \epsilon   \}
\end{equation*}
and we denote by $ K^+_\epsilon $ and $ K^-_\epsilon $ the convex bodies enclosed by $ \Sigma^+_\epsilon \cup P^+_\epsilon $ and $ \Sigma^-_\epsilon \cup P^-_\epsilon $ respectively. Then we define 
\begin{equation*}
	K_\epsilon = \{ x - \epsilon \eta : x \in K^+_{\epsilon}    \} \cup \{x + \epsilon \eta: x \in K^-_\epsilon \}.
\end{equation*}
Let $ X(x) = x $ for every $ x \in \mathbf{R}^{n+1} $. Since $ K^+_\epsilon $ is a set of finite perimeter, we denote by $ \mathbf{n}(K^+_\epsilon, \cdot) $ the exterior unit normal and we compute
\begin{flalign*}
(n+1)\Leb{n+1}(K^+_\epsilon)  &=  \int_{K^+_\epsilon} \divergence X \, d\Leb{n+1} \\
& = \int_{\Sigma^+_\epsilon \cup P^+_\epsilon} \mathbf{n}(K^+_\epsilon, x) \bullet X(x)\, d\Haus{n}x \\
& = \Haus{n}(\Sigma^+_\epsilon) - \epsilon\Haus{n}(P^+_\epsilon).
\end{flalign*}
We conclude that 
\begin{equation*}
	\frac{\Haus{n}(\partial K_\epsilon)}{(n+1)\Leb{n+1}(K_\epsilon)} = \frac{\Haus{n}(\Sigma^+_\epsilon)}{(n+1)\Leb{n+1}(K^+_\epsilon)} = 1 + \epsilon  \frac{\Haus{n}(P_\epsilon^+)}{(n+1)\Leb{n+1}(K^+_\epsilon)} > 1
\end{equation*}
for every $ 0 < \epsilon < 1 $. Finally we notice that the pointwise $ k $-th mean curvature of $ K_\epsilon $ equals $ {n \choose k} $ at $ \Haus{n} $ almost all points of $ \partial K_\epsilon $.
\end{Remark}

Here we generalize Theorem \ref{final theorem} proving that balls are the unique accumulation points of arbitrary convex bodies that satisfies the geometric condition on the $k$-th mean curvature asymptotically in $ L^1 $; see \eqref{compactness: eq1}.

\begin{Theorem}\label{compactness}
	Suppose $ (K_i)_{i \in \mathbf{N}} $ is a sequence of convex bodies in $ \mathbf{R}^{n+1} $ that converges in Hausdorff distance to a convex body $ K $, $ 1 \leq k \leq n $,
	\begin{equation*}
		\mu_i = \bigg( \frac{ \Haus{n}(\partial K_i)}{(n+1)\Leb{n+1}(K_i)}\bigg)^k {n \choose k}
	\end{equation*}
	 and 
	\begin{equation}\label{compactness: eq1}
	\lim_{i \to \infty} \| H_k(K_i, \cdot)- \mu_i \|_{L^1(\partial K_i)} =0.
	\end{equation}
	
	Then $ K $ is a ball.
\end{Theorem}
\begin{proof}
Since $ K_i $ converges in Hausdorff distance to $ K $, we employ \cite[1.8.20, 4.2.1]{MR3155183} to conclude that 
\begin{equation*}
	\lim_{i \to \infty}\Leb{n+1}(K_i) = \Leb{n+1}(K),
\end{equation*}
\begin{equation*}
	 C_m(K_i, \cdot) \overset{w}{\rightharpoonup} C_m(K, \cdot) \quad \textrm{for $ m = 0, \ldots , n $}
\end{equation*} 
and
\begin{equation*}
	\lim_{i \to \infty} \Haus{n}(\partial K_i) = \Haus{n}(\partial K).
\end{equation*}
Let $\mu = \Big( \frac{ \Haus{n}(\partial K)}{(n+1)\Leb{n+1}(K)}\Big)^k {n \choose k} $. It follows that $	\lim_{i \to \infty}\mu_i = \mu $ and we employ \eqref{absolutely continuous part of curvature meas} and \eqref{compactness: eq1} to obtain
\begin{equation*}
C_{n-k}^a(K_i, \cdot) \overset{w}{\rightharpoonup} \mu \Haus{n}\restrict \partial K.
\end{equation*}
Since $ C_{n-k}^a(K_i, \cdot) \leq C_{n-k}(K_i, \cdot) $ for every $ i \geq 1 $, we conclude that 
\begin{equation*}
C_{n-k}(K, \cdot) \geq \mu \Haus{n} \restrict \partial K.
\end{equation*}
Since $ C_{n-k}^s(K, \cdot) = C_{n-k}(K, \cdot)\restrict Z $ for a Borel subset $ Z \subseteq \partial K $ with $ \Haus{n}(Z) =0 $, see \cite[2.9.2]{MR0257325}, and 
\begin{equation*}
C_{n-k}^a(K, \cdot) - \mu \Haus{n}\restrict \partial K \geq - C_{n-k}^s(K, \cdot),
\end{equation*}
we conclude from \eqref{absolutely continuous part of curvature meas} that 
\begin{equation*}
	\int_{S} (H_k(K,x) - \mu)\, d\Haus{n}x \geq 0
\end{equation*}
for every Borel subset $ S \subseteq  \partial K $. Therefore $ H_k(K,x) \geq \mu $ for $ \Haus{n} $ a.e.\ $ x \in \partial K $ and the conclusion follows from \ref{final theorem}.
\end{proof}

We conclude this paper providing a solution of a conjectural statement in \cite[Remark 8]{MR3951441}. Notice that the bound on the measure of the singular set in Theorem \ref{another characterization of the sphere} is sharp, as the example of the two spherical caps in \ref{spherical caps} shows.

\begin{Definition}
Let $ K \subseteq \mathbf{R}^{n+1} $ be a convex body. 

We define $ \Reg(K) $ as the set of points $ x \in \partial K $ with the following property: there exists an open subset $ U $ of $ \mathbf{R}^{n+1} $ containing $ x $ such that $ U \cap \partial K  $ is a $ \mathcal{C}^{1,1} $-hypersuface.

We define $ \Sing(K) = \partial K \sim \Reg(K) $.
\end{Definition}

\begin{Theorem}\label{another characterization of the sphere}
	Suppose $ s \geq 0 $, $ k $ is an integer with $ 1 \leq k \leq n -s $ and $ K $ is a convex body in $ \mathbf{R}^{n+1} $ such that 
	\begin{equation*}
	\Haus{s}(\Sing(K)) =0.
	\end{equation*}
	
	Then 
	\begin{equation*}
C_{n-k}(K, B) = \int_B H_{k}(K,x)\, d\Haus{n}x
	\end{equation*}
for every Borel subset $ B \subseteq \partial K $. Additionally, if there exists $ \lambda \in \mathbf{R} $ such that $ H_k(K, x) = \lambda $ for $ \Haus{n} $ a.e.\ $ x \in \partial K $, then $ K $ is a ball.
\end{Theorem}

\begin{proof}
	We notice that 
	\begin{flalign*}
		C_{n-k}(K,B) & = \int_{N(K)|B} \overline{H}_{K,k}\, d\Haus{n}\\
		& = \int_{N(K)|B \cap \Reg(K)} \overline{H}_{K,k}\, d\Haus{n} + \int_{N(K)|B \cap \Sing(K)} \overline{H}_{K,k}\, d\Haus{n}.
	\end{flalign*}
	It follows from \cite[Theorem 5.5]{MR1742247} that 
	\begin{equation*}
	\int_{N(K)|B \cap \Sing(K)} \overline{H}_{K,k}\, d\Haus{n} = \int_{B \cap \Sing(K)}\Haus{k}(N(K,x))\, d\Haus{n-k}x =0.
	\end{equation*}
Recalling \ref{approx sff of a convex body}, we apply \ref{comparison of curvatures} with $ A $ and $ S $ replaced by $ K $ and $ \Reg(K) $ respectively, to conclude that there exists $ R \subseteq \Reg(K)  $ with $ \Haus{n}(\Reg(K) \sim R) =0 $ and 
\begin{equation*}
	H_k(K,x) = \overline{H}_{K,k}(x, u) \prod_{i=1}^{n}(1+\overline{\kappa}_{K,i}(x,u)^2)^{1/2}
\end{equation*}
for $ \Haus{n} $ a.e.\ $ (x,u) \in N(K)| R $. Since $ \Reg(K) $ is a $ \mathcal{C}^{1,1} $-hypersurface, we readily infer that 
\begin{equation*}
\Haus{n}(N(K)| (\Reg(K) \sim R)) =0.
\end{equation*}
Therefore, noting that $ \Reg(K)  \subseteq K^{(n)}_+ $ (see \eqref{C^n_+}), we apply \cite[5.4]{MR4117503} to conclude that 
\begin{equation*}
	\int_{N(K)|B \cap \Reg(K)} \overline{H}_{K,k}\, d\Haus{n} = \int_{B} H_k(K,x)\, d\Haus{n}x.
\end{equation*}

If there exists $ \lambda \in \mathbf{R} $ such that $ H_k(K,x) = \lambda $ for $ \Haus{n} $ a.e.\ $ x \in \partial K $ then 
\begin{equation*}
	C_{n-k}(K, B) = \lambda C_n(K, B)
\end{equation*}
for every Borel subset $ B  \subseteq \partial K $ and we conclude that $ K $ is a ball applying \cite[Satz 1.2]{MR522031}.
\end{proof}


\medskip 

\noindent Institut f\"ur Mathematik, Universit\"at Augsburg, \newline Universit\"atsstr.\ 14, 86159, Augsburg, Germany,
\newline mario.santilli@math.uni-augsburg.de
\end{document}